\documentclass[12pt, a4paper]{amsart}
\usepackage{graphicx}

\usepackage[a4paper, top=3cm, left=3cm, right=2cm, bottom=3cm]{geometry}
\usepackage{amsmath, amsthm}
\usepackage{amssymb}
\usepackage{multicol,comment}
\usepackage{amsfonts}
\usepackage{color, enumerate, fancyhdr, dsfont}
\usepackage[utf8]{inputenc}

\usepackage{float}
\usepackage[normalem]{ulem}

    \newcommand{\Q}{\mathbb{Q}}

    \newcommand{\R}{\mathbb{R}}

    \newcommand{\Cwf}{\mathcal{C}}

    \newcommand{\Iwf}{\mathcal{I}}
    \newcommand{\Jwf}{\mathcal{J}}
    \newcommand{\Kwf}{\mathcal{K}}
    
    \newcommand{\Mwf}{\mathcal{M}}
    \newcommand{\Nwf}{\mathcal{N}}
    
    \newcommand{\Pwf}{\mathcal{P}}

    \newcommand{\Swf}{\mathcal{S}}

    \newcommand{\minnon}{\mathrm{minnon}}

   \newcommand{\hgt}{\mathrm{ht}}

    \newcommand{\bfrak}{\mathfrak{b}}
    \newcommand{\cfrak}{\mathfrak{c}}
    \newcommand{\dfrak}{\mathfrak{d}}

    \newcommand{\add}{\mbox{\rm add}}
    \newcommand{\cov}{\mbox{\rm cov}}
    \newcommand{\non}{\mbox{\rm non}}
    \newcommand{\cof}{\mbox{\rm cof}}
    
    \newcommand{\limdir}{\mbox{\rm limdir}}

    \newcommand{\Ed}{\mathbf{Ed}}

    \newcommand{\Cor}{\mathbb{C}}

    \newcommand{\Loc}{\mathbb{LOC}}
    
    \newcommand{\Por}{\mathbb{P}}
    \newcommand{\Qor}{\mathbb{Q}}

    \newcommand{\Qnm}{\dot{\mathbb{Q}}}

    \newcommand{\cf}{\mbox{\rm cf}}

\title[On the strong measure zero ideal]{On cardinal characteristics associated with the strong measure zero ideal}

\author{Miguel A. Cardona}
\address{Institute of Discrete Mathematics and Geometry, TU Wien, Wiedner Hauptstrasse 8--10/104 A--1040 Wien, Austria.}
\email{miguel.montoya@tuwien.ac.at}
\urladdr{https://www.researchgate.net/profile/Miguel\_Cardona\_Montoya}

\thanks{This work was supported by the Austrian Science Fund (FWF) P30666. Recipent of a DOC Fellowship of the Austrian Academy of Sciences at the Institute of Discrete Mathematics and Geometry, TU Wien.}

\subjclass[2010]{03E17, 03E35, 03E40.}

\keywords{Strong measure zero sets, cardinal invariants, matrix iteration}

\begin{document}

\newcounter{enuAlph}
\renewcommand{\theenuAlph}{\Alph{enuAlph}}

\makeatletter
\def\@roman#1{\romannumeral #1}
\makeatother

\theoremstyle{plain}
  \newtheorem{theorem}{Theorem}[section]
  \newtheorem{corollary}[theorem]{Corollary}
  \newtheorem{lemma}[theorem]{Lemma}
  \newtheorem{mainlemma}[theorem]{Main Lemma}
  \newtheorem{prop}[theorem]{Proposition}
  \newtheorem{claim}[theorem]{Claim}
  \newtheorem{exer}[theorem]{Exercise}
  \newtheorem{question}[theorem]{Question}
   \newtheorem{fact}[theorem]{Fact}
  \newtheorem{problem}[theorem]{Problem}
  \newtheorem{conjecture}[theorem]{Conjecture}
  \newtheorem{Questions}[theorem]{Questions}
  \newtheorem*{thm}{Theorem}
  \newtheorem{teorema}[enuAlph]{Theorem}
  
  \newtheorem*{defn*}{Definition}
  \newtheorem*{corolario}{Corollary}
\theoremstyle{definition}
  \newtheorem{definition}[theorem]{Definition}
  \newtheorem{example}[theorem]{Example}
  \newtheorem{remark}[theorem]{Remark}
  \newtheorem{notation}[theorem]{Notation}
  \newtheorem{context}[theorem]{Context}

\newcommand{\azul}[1]{{\color{blue}#1}}
\newcommand{\rojo}[1]{{\color{red}#1}}
\newcommand{\tachar}[1]{{\color{red}\sout{#1}}}
\definecolor{amber}{rgb}{1.0,0.49,0.0}

\definecolor{ogreen}{RGB}{107,142,35}

\newcommand{\verde}[1]{{\color{ogreen}#1}}
\newcommand{\amber}[1]{{\color{amber}#1}}

\newcommand{\Fn}{\mathrm{Fn}}
\newcommand{\leqT}{\preceq_{\mathrm{T}}}
\newcommand{\eqT}{\cong_{\mathrm{T}}}
\newcommand{\la}{\langle}
\newcommand{\ra}{\rangle}
\newcommand{\id}{\mathrm{id}}
\newcommand{\Lv}{\mathrm{Lv}}
\newcommand{\sig}{\boldsymbol{\Sigma}}
\newcommand{\spl}{\mathrm{spl}}
\newcommand{\st}{\mathrm{st}}
\newcommand{\suc}{\mathrm{succ}}
\newcommand{\cosig}{\boldsymbol{\Pi}}
\newcommand{\Lb}{\mathrm{Lb}}
\newcommand{\pw}{\mathrm{pw}}
\newcommand{\Lc}{\mathbf{Lc}}

\newcommand{\SNcal}{\mathcal{SN}}
\newcommand{\Fr}{\mathrm{Fr}}
\newcommand{\Dbf}{\mathbf{D}}
\newcommand{\Cbf}{\mathbf{C}}
\newcommand{\Rbf}{\mathbf{R}}
\newcommand{\Sbf}{\mathbf{S}}
\newcommand{\Ibb}{\mathbb{I}}
\newcommand{\PTbb}{\mathbb{PT}}
\newcommand{\Qbb}{\mathbb{Q}}
\newcommand{\Tbb}{\mathbb{T}}
\newcommand{\Scal}{\mathcal{S}}

\newcommand{\sigmaf}{\sigma^f}
\newcommand{\Af}{A^f}

\begin{abstract} Let $\SNcal$ be the strong measure zero $\sigma$-ideal. We prove a result providing bounds for $\cof(\SNcal)$ which implies Yorioka's characterization of the cofinality of the strong measure zero. In addition, we use forcing matrix iterations to construct a model of ZFC that satisfies $\add(\SNcal)=\cov(\SNcal)<\non(\SNcal)<\cof(\SNcal)$. 
\end{abstract}

\maketitle

\section{Introduction}\label{SecIntro}
In this paper we continue the study of \cite{CMR} on the cardinal characteristics of the continuum associated with the ideal of strong measure zero sets. In general these cardinals are defined as follows. Let $\Iwf$ be an ideal on $\Pwf(X)$ containing all the finite subsets of $X$. Define \emph{the cardinal characteristics associated with $\Iwf$} by:
\begin{align*}
    \add(\Iwf)&:=\min\{|\Jwf|:\Jwf\subseteq\Iwf\text{\ and } \bigcup\Jwf\notin\Iwf\} \emph{\ the additivity of $\Iwf$};\\
    \cov(\Iwf)&:=\min\{|\Jwf|:\Jwf\subseteq\Iwf\text{\ and }\bigcup\Jwf=X\} \emph{\ the covering of $\Iwf$};\\
    \non(\Iwf)&:=\min\{|A|:A\subseteq X\text{\ and }A\notin\Iwf\} \emph{\ the uniformity of $\Iwf$};\\
    \cof(\Iwf)&:=\min\{|\Jwf|:\Jwf\subseteq\Iwf\text{\ is cofinal in }\la\Iwf,\subseteq\ra\} \emph{\ the cofinality of $\Iwf$}.
\end{align*}


    
    
Figure \ref{Figaddetc} shows the ``trivial'' inequalities between the cardinal characteristics associated with $\Iwf$.

\begin{figure}[H]
  \begin{center}
    \includegraphics[scale=1.0]{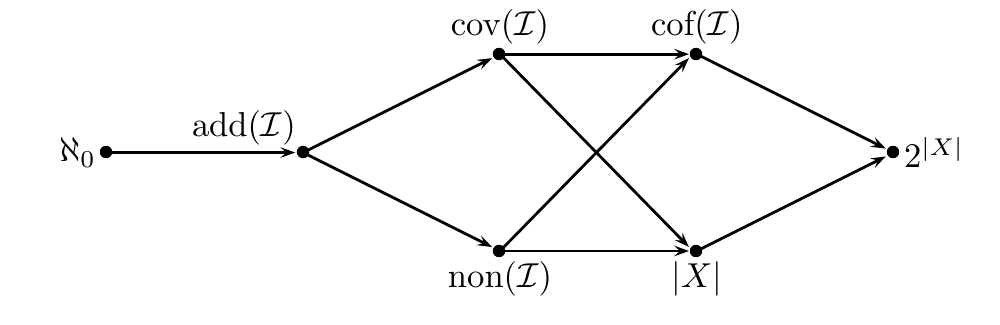}
    \caption{Cardinal characteristics associated with $\Iwf$. An arrow  $\mathfrak x\rightarrow\mathfrak y$ means that (provably) 
    $\mathfrak x\le\mathfrak y$.}
    \label{Figaddetc}
  \end{center}
\end{figure}
\vspace{-0.35cm}
Classical examples are the cardinal characteristics in Cicho\'n's diagram (see Figure \ref{Figcichon}), which is composed by the cardinal characteristics associated with $\Mwf$, $\Nwf$, $\Kwf$ and $\Cwf$, where  $\Mwf$ is the family of meager subsets of $\R$, $\Nwf$ is the family of Lebesgue measure zero subsets of $\R$, $\Kwf$ is the $\sigma$-ideal generated by the subsets of $\R$ whose intersection with $\Q^*$ (the set of irrational numbers) is compact in $\Q^*$, and $\Cwf$ is the $\sigma$-ideal of countable subsets of $\R$. It is known that $\add(\Kwf)=\non(\Kwf)=\bfrak$, $\add(\Cwf)=\non(\Cwf)=\aleph_1$, $\cov(\Kwf)=\cof(\Kwf)=\dfrak$, and $\cov(\Cwf)=\cof(\Cwf)=\cfrak$, where $\bfrak$, $\dfrak$ and $\cfrak$ are the bounding number, dominating number and the size of $\R$, respectively.

Borel \cite{Borel} introduced the notion of \emph{strong measure zero sets} (see Definition \ref{defsmz}). Borel \cite{Borel} conjectured that each subset of the real line that has strong measure zero is countable, which is known as Borel's Conjecture (BC). Sierpi\'nski  \cite{S} showed that the Continuum Hypothesis implies the existence of an uncountable set of real numbers of strong measure zero, and  Laver \cite{LaverBorel} proved the consistency of BC with ZFC by forcing, i.e, BC cannot be proven nor refuted in ZFC.

\begin{figure}
  \begin{center}
    \includegraphics[scale=1.0]{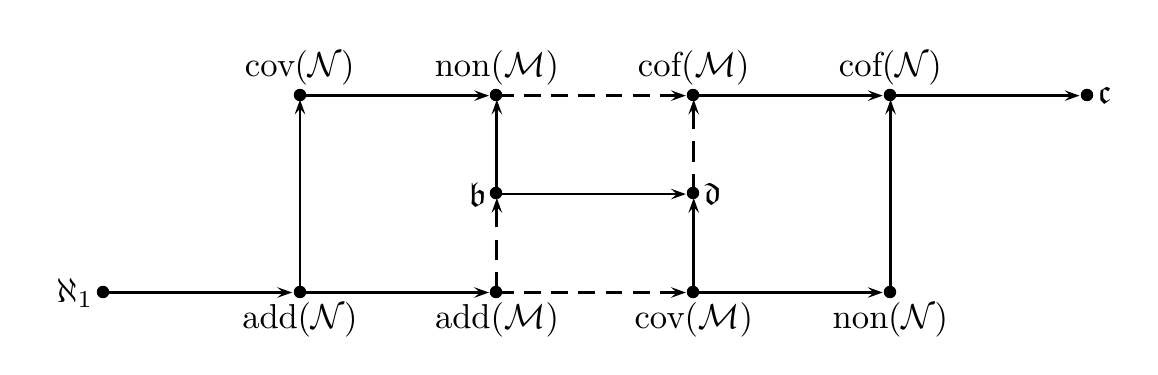}
    \caption{Cicho\'n's diagram.
    $\mathfrak x\rightarrow\mathfrak y$ means that (provably) 
    $\mathfrak x\le\mathfrak y$, and the dashed arrows indicate that 
    $\add(\Mwf)=\min\{\bfrak,\cov(\Mwf)\}$ and $\cof(\Mwf)=\max\{\dfrak,\non(\Nwf)\}$.}
     \label{Figcichon}
  \end{center}
\end{figure}

The cardinal characteristics associated with the ideal of strong measure zero sets have been interesting objects of research, in particular when related to the cardinals in Cicho\'n's diagram.  Denote by $\SNcal$ the ideal of strong measure zero subsets of $\R$.

The following holds in ZFC:
\begin{enumerate}[({S}1)]
    \item  (Carlson \cite{Carlson}) $\add(\Nwf)\leq\add(\SNcal)$,
    \item  $\cov(\Nwf)\leq\cov(\SNcal)\leq\cfrak$ ,
    \item (Miller \cite{Miller}) $\cov(\Mwf)\leq\non(\SNcal)\leq\non(\Nwf)$ and $\add(\Mwf)=\min\{\bfrak,\non(\SNcal)\}$,
    \item  (\cite{Osuga}) $\cof(\SNcal)\leq 2^\dfrak$.
\end{enumerate}
On the other hand, the following inequalities are \emph{consistent with ZFC}:
\begin{enumerate}[({C}1)]
    \item (Goldstern, Judah and Shelah \cite{GJS}) $\cof(\Mwf)<\add(\SNcal)$ ,
    \item (Pawlikowski \cite{P90}) $\cov(\SNcal)<\add(\Mwf)$, 
    \item  $\cfrak<\cof(\SNcal)$ (follows from CH),
    \item (\cite{Yorioka}) $\cof(\SNcal)<\cfrak$,
    \item (\cite{CMR}) $\non(\SNcal)<\cov(\SNcal)<\cof(\SNcal)$.
    \item (\cite{CMR}) $\cof(\Nwf)<\cov(\SNcal)$.
\end{enumerate}

Yorioka introduced a characterization of $\SNcal$ in terms of $\sigma$-ideals $\Iwf_f$, parametrized by functions $f\in\omega^\omega$, which we call \textit{Yorioka ideals} (see Definition \ref{defidealyorioka}). More concretely, $\SNcal=\bigcap\{\Iwf_f\ |\ f\in\omega^\omega\textrm{\ increasing}\}$ and $\Iwf_f\subseteq\Nwf$. Figure \ref{Extended} summarizes the relationship between the cardinal invariants associated with Yorioka ideals and the cardinals in Cicho\'n's diagram, see e.g. \cite{ KO,Osuga, CM}. 

Yorioka also gave a important description of $\cof(\SNcal)$, namely

\begin{theorem}[{\cite[Thm. 2.6]{Yorioka}}]\label{yoriokachac}
If $\add(\Iwf_f)=\cof(\Iwf_f)=\kappa$ for all increasing $f\in\omega^\omega$ then $\cof(\SNcal)=\dfrak_\kappa$ (the \textit{dominating number of} $\kappa^\kappa$)\footnote{Original statement abreviated thanks to Figure \ref{Extended}}.
\end{theorem}

To prove Theorem \ref{yoriokachac}, Yorioka constructed a dominating family $\la f_\alpha\ |\ \alpha<\kappa\ra$ along with a \emph{matrix $\la A_\alpha^\beta:\alpha, \beta<\kappa\ra$} of subsets of the Cantor space $2^\omega$ fulling the following properties: 
\begin{itemize}
    \item[(i)] $\forall{\alpha, \beta<\kappa}(A_\alpha^\beta\subseteq 2^\omega\textrm{\ is a dense }G_\delta\textrm{\ set\ }\textrm{and\ } \,A_\alpha^\beta\in\Iwf_{f_\alpha})$;
    \item[(ii)] $\forall{\alpha, \beta, \beta'<\kappa}(\beta\leq\beta'\to A_\alpha^\beta\subseteq A_\alpha^{\beta'})$;
    \item[(iii)] $\forall{\alpha<\kappa}\forall{A\in\Iwf_{f_\alpha}}\exists{\beta<\kappa}(A\subseteq A_\alpha^\beta)$; and 
    \item[(iv)] $\forall{\alpha<\kappa}\forall{B\in\Iwf_{f_\alpha}}(\alpha>0\to\bigcap_{\gamma<\alpha}A_\gamma^0\smallsetminus B\neq \varnothing)$.
    \end{itemize}
 This gives a Tukey isomorphism between $\SNcal$ and $\la\kappa^\kappa,\leq\ra$ (where $\leq$ is interpreted as pointwise). 
 
 In this paper, the author introduces the notion of \textit{dominating system} (see Definition \ref{defsuitmatrix}), which improves this construction and refines Theorem \ref{yoriokachac} by providing bounds to $\cof(\SNcal)$ without the hypothesis $\add(\Iwf_f)=\cof(\Iwf_f)$ for all $f$. Concretely the author proves the following main result.
 
\begin{teorema}[Theorems \ref{lemfamily} and \ref{lemfamilytwo}]\label{mainth1}
If there is some $\lambda$-dominating system on a directed set $\la S,\leq_S\ra$ then  
\begin{itemize}
    \item[(i)] $\SNcal\leqT \la S^\lambda,\leq\ra$.
    \item[(ii)] If $\minnon\geq\lambda$ and $\la S,\leq_S\ra=\la\kappa\times\lambda,\leq\ra$ with $\kappa\leq\lambda$, then $\la\lambda^\lambda,\leq\ra\leqT\SNcal$.
\end{itemize}
\end{teorema}

The author with Mej\'ia and Rivera-Madrid \cite[Section 5]{CMR} asks the following questions: Is it consistent with ZFC that 
\begin{enumerate}
    \item[(Q1)]  $\add(\SNcal)<\min\{\cov(\SNcal),\non(\SNcal)\}$?
    \item[(Q2)]  $\add(\SNcal)<\non(\SNcal)<\cov(\SNcal)<\cof(\SNcal)$?   
    \item[(Q3)]  $\add(\SNcal)<\cov(\SNcal)<\non(\SNcal)<\cof(\SNcal)$?  
\end{enumerate} 

Question (Q2) was answered partially by the author with Mej\'ia and Rivera-Madrid \cite{CMR}. Concretely, they showed that, in Sack's model,
\[\add(\SNcal)=\non(\SNcal)=\aleph_1<\cov(\SNcal)=\aleph_2=\cfrak<\cof(\SNcal).\]
This is the first result where more than two cardinal invariants associated with $\SNcal$ are pairwise different. 

In this work, we partially answer question (Q3). More concretely, we prove the following.

\begin{teorema}[Theorem \ref{maintheorem}]\label{introtheorem}\label{mainth2}
   Let $\kappa\leq\lambda$ be regular uncountable cardinals where $\kappa^{<\kappa}=\kappa$ and let $\lambda_1,\lambda_2$ be cardinals such that $\lambda^{<\lambda}=\lambda$, $\lambda\leq\lambda_1$, $\lambda_2^{\lambda}=\lambda_2$ and $\lambda_1^{\aleph_0}=\lambda_1$. Then there is a cofinality preserving poset that forces 
\[\add(\SNcal)=\cov(\SNcal)=\kappa\leq\non(\SNcal)=\lambda\leq\cof(\SNcal)=\lambda_2
   \textrm{\ and\ }\cfrak=\lambda_1\]

\end{teorema}

\begin{figure}
  \begin{center}
    \includegraphics[scale=0.9]{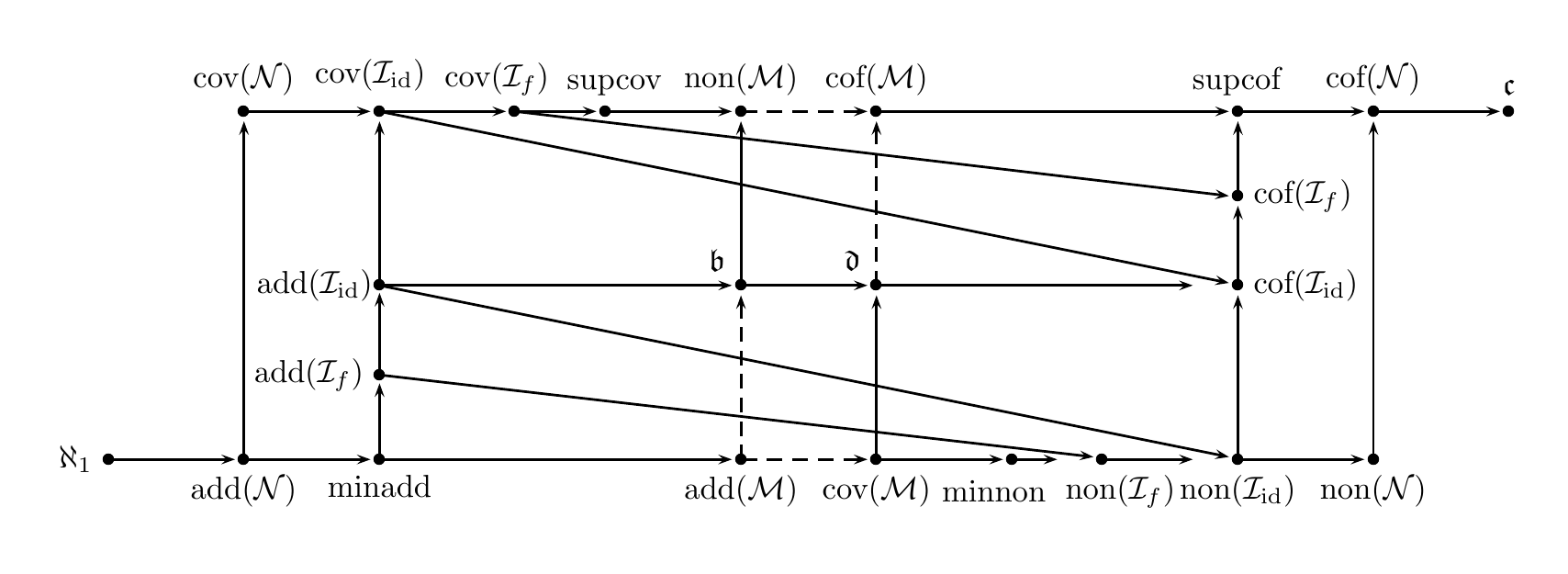}
    \caption{Extended Cicho\'n's diagram.}
    \label{Extended}
  \end{center}
\end{figure}

This is the second result where more than two cardinal invariants associated with $\SNcal$ are pairwise different. 

Now, to prove Theorem \ref{introtheorem}, we use the method of matrix forcing iterations. To achieve this, we go through the following steps:
\begin{itemize}
    \item[(P1)] We will force $\cfrak=\lambda_1$ and $\dfrak_{\kappa\times\lambda}^{\lambda}=\dfrak_{\lambda}^{}=\lambda_2$ by generalized Cohen forcing. These cardinals are introduced in Section \ref{SecPre}.
    \item[(P2)] Afterwards, we construct the matrix. Along the matrix, we will construct a dominating family $\la f_\gamma\ |\ \gamma<\lambda\ra$ along with a $\lambda$-dominating system on $\la\kappa\times\lambda,\leq\ra$. Thanks to Theorem \ref{mainth1}, the matrix forces  $\cof(\SNcal)=\lambda_2$. For the construction, we use restricted localization forcing.
    \item[(P3)] The constructed matrix forces $\cov(\Mwf)=\cof(\Nwf)=\lambda$ and $\add(\Nwf)=\non(\Mwf)=\kappa$, so $\kappa\leq\add(\SNcal)$ and $\non(\SNcal)=\lambda$ by (S1) and (S3). Since the matrix is obtained by a FS iteration of length with cofinality $\kappa$, $\cov(\SNcal)\leq\kappa$.

\end{itemize}

This paper is structured as follows. We review in Section 2 the basic notation and the results this paper is based on. The notions of $\Iwf_f$ directed system and $\lambda$-dominating system are introduced in Section 3, as well as the proof of Theorem \ref{mainth1}. In Section 4 we prove Theorem \ref{mainth2}. Finally, in Section 5 we present some open questions.

\section{Preliminaries}\label{SecPre} We start with the following basic notions. Let $\kappa$ be an infinite cardinal. Denote by $\Fn_{<\kappa}(I,J)$ the poset of partial functions from $I$ into $J$ with domain of size $<\kappa$, ordered by $\supseteq$. If $z$ is an ordered pair, $z_0$ and $z_1$ denotes the first and second component of $z$ respectively. Set $\omega^{\uparrow\omega}:=\{d\in\omega^{\omega}:d(0)=0\textrm{\ and }d\textrm{\ is increasing}\}$. For any set $A$, $\id_{A}$ denotes the \textit{identity function on $A$}. For each $\sigma \in (2^{<\omega})^{\omega}$ define $\hgt_{\sigma}\in\omega^{\omega}$ by $\hgt_{\sigma}(i):=|\sigma(i)|$ (see Figure \ref{diagram}). 

\begin{figure}
\begin{center}\vspace{1cm}
\includegraphics[width=1\linewidth]{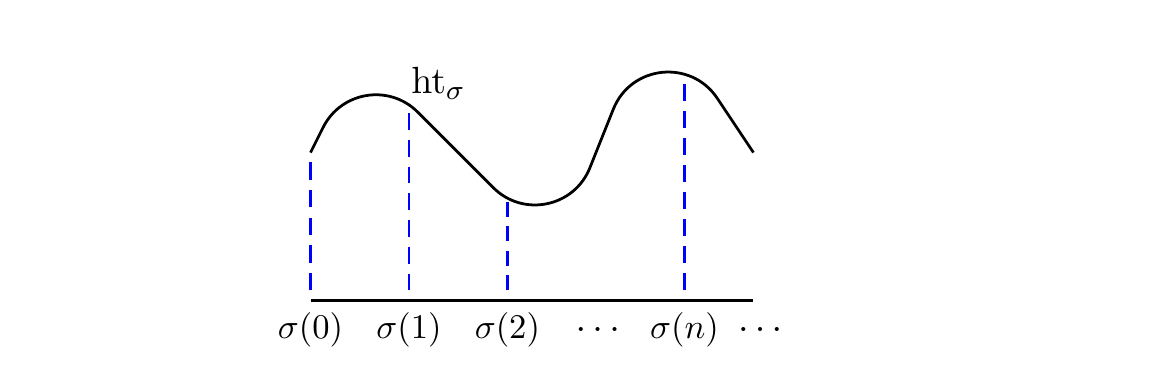}
\caption{Functions $\sigma$ and $\hgt_\sigma$.}
\label{diagram}
\end{center}
\end{figure}

Typically, cardinal invariants of the continuum are defined through relational systems
as follows. A \textit{relational system} is a triplet $\Rbf=\la X, Y,R\ra$ where $R$ is a relation contained in $X\times Y$. For $x\in X$ and $y\in Y$, $x R y$ is often read $y$ $R$-\textit{dominates} $x$.
 A family $E\subseteq X$ is \textit{$\Rbf$-bounded} if $\exists y\in Y\forall x\in E(x R y)$. Dually, $D\subseteq Y$ is \textit{$\Rbf$-dominating} if $\forall x\in X\exists y\in D(x R y)$. Such a relational system has two cardinal invariants associated with it: 

\vspace{-0.55cm}

\begin{align*}
\bfrak(\Rbf)&:=\min\{|E|:E\subseteq X  \textrm{\ is\ }  \Rbf\textrm{-unbounded}\},\\
\dfrak(\Rbf)&:=\min\{|D|:D\subseteq Y \textrm{\ is\ } \Rbf\textrm{-dominating}\}.
\end{align*}

Let $\Rbf':=\la X',Y',R'\ra$ be another relational system. If there are maps $\Psi_1:X\to X'$ and $\Psi_2:Y'\to Y$ such that, for any $x\in X$ and $y'\in Y'$, if $\Psi_1(x) R'y'$ then $x R \Psi_2(y')$, we say that \emph{$\Rbf$ is Tukey below $\Rbf'$}, denoted by $\Rbf\leqT\Rbf'$. Say that \emph{$\Rbf$ and $\Rbf'$ are Tukey equivalent}, denoted by $\Rbf\eqT\Rbf'$, if $\Rbf\leqT\Rbf'$ and $\Rbf'\leqT\Rbf$. Note that $\Rbf\leqT\Rbf'$ implies $\bfrak(\Rbf')\leq\bfrak(\Rbf)$ and $\dfrak(\Rbf)\leq\dfrak(\Rbf')$.

Let $\Rbf:=\la X,Y,R\ra$ and $\Rbf':=\la X',Y',R'\ra$  be two relational systems. Set $\Rbf\otimes\Rbf':=\la X\times X',Y\times Y',R_{\otimes}\ra$, where $(x,x')R_{\otimes}(y,y')$ iff $x R x'$ and $y R' y'$.

\begin{fact}[{\cite[Thm 4.11]{blass}}]
$\max\{\dfrak(\Rbf),\dfrak(\Rbf')\}\leq\dfrak(\Rbf\otimes\Rbf')\leq\dfrak(\Rbf)\cdot\dfrak(\Rbf')$ and $\bfrak(\Rbf\otimes\Rbf')=\min\{\bfrak(\Rbf),\bfrak(\Rbf')\}$
\end{fact}

A \emph{directed set} is a set $S$ with a preorder $\leq_S$ such that every finite subset of $S$ has an upper bound. In other words, for any $x$ and $y$ in $S$ there exists a $z$ in S with $x\leq_S z$ and $y\leq_S z$.

Given a function $b$ with domain $\omega$ such that $b(i)\neq\emptyset$ for all $i<\omega$, and $h\in \omega^\omega$, define $\Scal(b,h)=\prod_{n<\omega}[b(n)]^{\leq h(n)}$. For $x\in\omega^\omega$ and $\psi\in\Scal(b,h)$, say that $x\in^{*}\varphi$ iff $\forall^{\infty}{n<\omega}(x(n)\in \varphi(n))$, which is read \textit{$\varphi$ localizes $x$}.

\begin{example}\label{exam} Let $\kappa$ and $\lambda$ be non-zero cardinals, and let $\la S,\leq_S \ra$ be a directed set.
\begin{enumerate}[(1)]
    \item Consider the relational system $\Ed:=\la\omega^\omega,\omega^\omega,\neq^*\ra$, where for $f,g\in\omega^\omega$, $f\neq^* g$ iff $\exists n<\omega\forall m\geq n(f(m)\neq g(m))$. By \cite[Thm. 2.4.1 \& Thm. 2.4.7]{BJ},  $\bfrak(\Ed)=\non(\Mwf)$ and $\dfrak(\Ed)=\cov(\Mwf)$.
    \item Define $\Lc(\omega,h):=\la \omega^\omega,\Scal(\omega,h),\in^*\ra$ (here, $\omega$ denotes the constant function with value $\omega$), which is a relational system. If $h\in\omega^\omega$ goes to infinity then $\add(\Nwf):=\bfrak(\Lc(\omega,h))$ and $\cof(\Nwf):=\dfrak(\Lc(\omega,h))$ (see \cite[Thm. 2.3.9]{BJ}).
    \item As a relational system, $S$ is $\la S, S,\leq_S\ra$, 
    $\bfrak(S):=\bfrak(\la S,S,\leq_S\ra)$ and $\dfrak(S):=\dfrak(\la S,S,\leq_S\ra)$. Note that, if $S$ has no maximum, then $\bfrak(S)$ is regular and $\bfrak(S)\leq\dfrak(S)$. Even more, $\cf(\dfrak(S))\geq\bfrak(S)$.
    \item Consider the relational system $\mathbf{D}_{S}^\lambda:=\la S^\lambda,S^\lambda,\leq \ra$ where $X\leq Y$ iff $\forall \alpha<\lambda(X(\alpha)\leq_S Y(\alpha))$. Define $\bfrak_{S}^\lambda:=\bfrak(\mathbf{D}_{S}^\lambda)$ and $\dfrak_{S}^\lambda:=\dfrak(\mathbf{D}_{S}^\lambda)$. 
    \item Denote $\bfrak_{\kappa\times\lambda}^\lambda:=\bfrak(\Dbf_{\kappa\times\lambda}^\lambda)$ and $\dfrak_{\kappa\times\lambda}^\lambda:=\dfrak(\Dbf_{\kappa\times\lambda}^\lambda)$ where $\lambda\times\kappa$ is ordered by $(\alpha,\beta)\leq(\alpha',\beta')$ iff $\alpha\leq\alpha'$ and $\beta\leq\beta'$.
    \item Assume that $\lambda$ is infinite. Consider the relational system $\Dbf_{S}^{\lambda}(\leq^*):=\la S^\lambda,S^\lambda,\leq^* \ra$ where $X\leq^* Y$ iff $\exists\alpha<\lambda\forall\beta\in[\alpha,\lambda)(X(\beta)\leq Y(\beta))$. Set $\bfrak_{S}^\lambda(\leq^*):=\bfrak(\mathbf{D}_{S}^\lambda(\leq^*))$ and $\dfrak_{S}^\lambda(\leq^*):=\dfrak(\mathbf{D}_{S}^\lambda(\leq^*))$.
    \item When $\kappa$ is infinite, define $\bfrak_\kappa:=\bfrak_\kappa^\kappa(\leq^*)$ and $\dfrak_\kappa:=\dfrak_\kappa^\kappa(\leq^*)$ (this is a particular case of $\Dbf_S^\lambda(\leq^*)$ with $S=\lambda=\kappa$). These are the well known \textit{unbounding number of}\,\,$\kappa^\kappa$ and  \textit{ dominating number of}\,\, $\kappa^\kappa$ respectively.
\end{enumerate}
\end{example}

The following result follows from (3).


\begin{corollary}
Let $S$ be a directed partial order and let $\lambda$ be a non-zero cardinal. If $S$ has no maximum then $\aleph_0\leq \cf(\bfrak_S^\lambda)=\bfrak_S^\lambda\leq \cf(\dfrak_S^\lambda)\leq \dfrak_S^\lambda\leq |S|^\lambda$.
\end{corollary}

We prove some results about the cardinal invariants associated with $\Dbf_S^\lambda$ and $\Dbf_S^\lambda(\leq^*)$.

\begin{lemma}\label{basic}
Let $S$ be a directed partial order and let $\lambda$ be a non-zero cardinal. If $S$ has no maximum then   
\begin{itemize}
    \item[(i)] $\bfrak(S)=\bfrak_{S}^{\lambda}\leq \dfrak(S)\leq\dfrak_{S}^\lambda\leq\dfrak(S)^\lambda\leq|S|^\lambda$. Even more, $S\leqT\Dbf_S^\lambda$. 
    \item[(ii)] If $\lambda<\bfrak(S)$ then $\Dbf_S^\lambda\eqT S$.
    \item[(iii)] If $\lambda\leq\lambda'$ are non-zero cardinals, then $\mathbf{D}_{S}^\lambda\leqT\mathbf{D}_{S}^{\lambda^{\prime}}$. In particular, $\dfrak_{S}^{\lambda}\leq\dfrak_{S}^{\lambda^{\prime}}$.
    \item[(iv)] If $\lambda$ is infinite then $\dfrak_S^\lambda(\leq^*)\leq\dfrak_S^\lambda$ and $\bfrak(S)\leq\bfrak_S^\lambda(\leq^*)$. 
\end{itemize}
\end{lemma}

\begin{proof} \begin{itemize}
\item[(i)] Clearly $\dfrak_{S}^\lambda\leq\dfrak(S)^\lambda\leq |S|^\lambda$ and $\bfrak(S)\leq\dfrak(S)$. It remains to prove that $\bfrak(S)=\bfrak_{S}^{\lambda}$ and $\dfrak(S)\leq\dfrak_S^\lambda$. To see $\bfrak(S)\leq\bfrak_S^\lambda$, let $B\subseteq S^\lambda$ with $|B|<\bfrak(S)$. For $\zeta<\lambda$, define $\Gamma_\zeta:=\{f(\zeta)\,|\,f\in B\}$. Since $|\Gamma_\zeta|<\bfrak(S)$, choose $i_\zeta\in S$ such that $f(\zeta)\leq_{S}i_\zeta$ for all $f\in B$. Define $g\in S^\lambda$ by $g(\zeta):=i_\zeta$. Then  $g$ bounds $B$. 

For the converse inequality, it suffices to prove that $S\leqT\Dbf_S^\lambda$. For $f\in S^\lambda$ put $\Psi_2(f):=f(0)$. On the other hand, for $i\in S$ define $f_i\in S^\lambda$ by $f_i(\xi):=i$ for each $\xi<\lambda$, so put $\Psi_1(i):=f_i$. It is clear that if $f_i\leq f$ then $i\leq_S f(0)$. Hence, $\dfrak(S)\leq\dfrak_S^\lambda$ and $\bfrak_S^\lambda\leq\bfrak(S)$.


\item[(ii)] By (i), it is enough to show that $\Dbf_S^\lambda\leqT S$. To this end let $f\in S^\lambda$. Define $D:=\{f(\beta):\beta<\lambda\}$. Since $|D|<\lambda$, choose $i_f$ in $S$ such that $f(\beta)\leq_S i_f$ for each $\beta<\lambda$ and put $\Psi_1(f):=i_f$. 

Finally, put $\Psi_2(j):=f_j$ for $j\in S$ as in (i). It remains to check that, if $i_f\leq_S j$ then $f\leq f_j$. Fix $\beta<\lambda$. Then $f(\beta)\leq_S i_f\leq_S j=f_j(\beta)$. 

\item[(iii)] Define $\Psi_2:S^{\lambda'}\to S^\lambda$ as follows: For $f\in S^{\lambda'}$ set $\Psi_2(f):=f\upharpoonright\lambda$. 

To define $\Psi_1:S^\lambda\to S^{\lambda'}$,  for $g\in S^\lambda$ define $g^*\in S^{\lambda'}$ by setting, for any $\xi<\lambda'$, $g^*(\xi):=g(\xi)$ if $\xi<\lambda$, and $g^*(\xi)=0$ otherwise. Put $\Psi_1(g):=g^*$. It is clear that if $g^*\leq f$ then $g\leq f\upharpoonright\lambda$. 
\item[(iv)] Obvious because $\Dbf_S^\lambda(\leq^*)\leqT \Dbf_S^\lambda$. 
\qedhere
\end{itemize}
\end{proof}

\begin{lemma}\label{lemmd}
If $\lambda$ is an infinite cardinal and $S$ has no maximum, then $\dfrak_S^\lambda(\leq^*)>\lambda$. 
\end{lemma}
\begin{proof}
Let $F:=\{f_\xi\,|\,\xi<\lambda\}\subseteq S^\lambda$ be a family of size $\lambda$, and let $K$ be a bijection from $\lambda$ onto $\lambda\times\lambda$. Define $f\in S^\lambda$ as follows: for any $\beta<\lambda$ we can choose $s_\beta>f_{K(\beta)_0}(\beta)$ (such $s_\beta$ exists because $S$ has no maximum) and put $f(\beta):=s_\beta$. For each $\xi, \eta<\lambda$ set $\beta_{\xi,\eta}:=K^{-1}(\xi,\eta)$, so $K(\beta_{\xi,\eta})=\xi$ and $f(\beta_{\xi,\eta})>f_\xi(\beta_{\xi,\eta})$. Then $|\{\beta<\lambda\,|\,f(\beta)>f_\xi(\beta)\}|=\lambda$. 
\end{proof}

In the next theorem we give a characterization of $\dfrak_S^\lambda$. 

\begin{theorem}\label{thmd}
If $\lambda$ is an infinite cardinal and $S$ has no maximum, then \[\dfrak_S^\lambda=\dfrak_S^\lambda(\leq^*)\cdot\sup_{\kappa<\lambda}\{\dfrak_S^\kappa\}.\]
\end{theorem}
\begin{proof}
 Clearly $\dfrak_S^\lambda(\leq^*)\cdot\sup_{\xi<\lambda}\{\dfrak_S^{|\xi|}\}\leq\dfrak_S^\lambda$ because $\dfrak_S^{|\xi|}\leq\dfrak_S^\lambda$ (Lemma \ref{basic}(ii))  and $\dfrak_S^\lambda(\leq^*)\leq\dfrak_S^\lambda$.
 
For $\xi<\lambda$, choose $D_\xi\subseteq S^\xi$ $\leq$-dominating with $|D_\xi|=\dfrak_S^{|\xi|}$. Take a $\leq^*$-dominating family $D\subseteq S^\lambda$. For $g\in D$ and $h\in D_\xi$ with $\xi<\lambda$ define the function $f_{g,h}\in S^\lambda$ by
 \[ f_{g,h}(\beta)
   := 
     \begin{cases}
     g(\beta) & \emph{if $\beta\geq\xi$,}\\
       h(\beta) & \emph{if $\beta<\xi$.}\\
     \end{cases}
\] 
Since $|\{f_{g,h}\,|\,g\in D\,\wedge\, \exists \xi<\lambda(h\in D_\xi)\}|\leq \dfrak_S^\lambda(\leq^*)\cdot\sup_{\xi<\lambda}\{\dfrak_S^{|\xi|}\}\cdot\lambda=\dfrak_S^\lambda(\leq^*)\cdot\sup_{\xi<\lambda}\{\dfrak_S^{|\xi|}\}$ by Lemma \ref{lemmd}, it sufficies to prove that this family is $\leq $-dominating. To this end let $f\in S^\lambda$. Find $g\in D$ and $\xi<\lambda$ such that $f(\beta)\leq g(\beta)$ for all $\beta\geq\xi$. Then, for $\beta<\xi$ set $h_\xi(\beta):=\max\{f(\beta),g(\beta)\}\in S^\xi$, so there is some $h\in D_\xi$ such that $h_\xi\leq h$. Therefore, $f_{g,h}$ dominates $f$ everywhere.  
\end{proof}

It is known that $\dfrak_\lambda^\lambda=\dfrak_\lambda$ when $\lambda$ is regular, even more, this follows from Theorem \ref{thmd} because $\dfrak_\lambda^\kappa=\lambda$ when $\kappa<\lambda$. However, $\dfrak_\lambda^\lambda=\dfrak_{\cf(\lambda)}^\lambda$ in general. More details about $\dfrak_\kappa^\lambda$ can be found in \cite{brenhigher}.

\begin{lemma}\label{lemax}
$\Dbf_{\kappa\times\lambda}^{\lambda}\eqT\Dbf_\kappa^\lambda\otimes\Dbf_\lambda^\lambda$. In particular, $\dfrak_{\kappa\times\lambda}^{\lambda}=\max\{\dfrak_{\kappa}^{\lambda},\dfrak_\lambda^\lambda\}$.
\end{lemma}
\begin{proof}
To define $\Psi_1:(\kappa\times\lambda)^\lambda\to\kappa^\lambda\times\lambda^\lambda$, for  $F\in(\kappa\times\lambda)^\lambda$ define $f_F\in\kappa^\lambda$ and $g_F\in\lambda^\lambda$ by setting $f_F(\xi):=F(\xi)_0$ and $g_F(\xi):=F(\xi)_1$. Put $\Psi_1(F):=(f_F,g_F)$. 
Now, define $\Psi_2:\kappa^\lambda\times\lambda^\lambda\to(\kappa\times\lambda)^\lambda$ as follows: For $f\in\kappa^\lambda$ and $g\in\lambda^\lambda$ set $F_{f,g}\in(\kappa\times\lambda)^\lambda$ by setting, for any $\zeta<\lambda$, $F_{f,g}(\zeta)=(f(\zeta),g(\zeta))$. Put $\Psi_2(f,g):=F_{f,g}$. It is clear that if $(f_F,g_F)\leq_\otimes(f,g)$ then $F\leq_{\kappa\times\lambda}F_{f,g}$. Also, if $F_{f,g}\leq_{\kappa\times\lambda}F$ then $(f,g)\leq_\otimes(f_F,g_F)$.
\end{proof}






\begin{definition}
Let  $\gamma, \pi$ be ordinals. A \textit{simple matrix iteration} $\Por=\langle\Por_{\alpha, \xi},\dot{\Qor}_{\alpha,\xi}\,|\,\,\alpha\leq\gamma, 0\leq\xi\leq\pi\rangle$ fullfils the following requirements.
\begin{enumerate}
    \item[(i)] $\cof(\gamma)>\omega$, 
    \item[(ii)] $\Qnm_{\alpha,0}=\Por_{\alpha, 1}=\Cor_\alpha:=\mathrm{Fn}(\alpha\times\omega,2)$, 
    \item[(iii)] for each $0<\xi<\pi$, $\Delta(\xi)<\gamma$ is non-limit and $\dot{\Qbb}_\xi$ is a $\Por_{\Delta(\xi),\xi}$-name of a poset such that $\Por_{\gamma,\,\xi}$ forces it to be ccc, and
    \item[(iv)] $\Por_{\alpha,\xi+1}=\Por_{\alpha,\,\xi}\ast\dot{\Qbb}_{\alpha,\xi}$, where
\end{enumerate}  
\[\dot{\mathbb{Q}}_{\alpha,\xi}:=
\begin{cases}
    \Qnm_\xi & \textrm{if  $\alpha\geq\Delta(\xi)$,}\\
    \mathds{1}         & \textrm{otherwise,}
\end{cases}\]
\begin{enumerate}
    \item[(v)] for $\xi$ limit, $\Por_{\alpha,\xi}=\limdir_{\eta<\xi}\Por_{\alpha,\eta}$. 
\end{enumerate}
\end{definition}

As a consequence, $\alpha\leq\beta\leq\gamma$ and $\xi\leq\eta\leq\pi$ imply $\Por_{\alpha,\xi}\lessdot\Por_{\beta,\eta}$. 

\begin{lemma}[{See e.g. \cite[Cor. 2.6]{mejiavert}}]\label{realint}
 Assume that $\Por=\langle\Por_{\alpha, \xi},\dot{\Qor}_{\alpha,\xi}\,|\,\,\alpha\leq\gamma, \xi\leq\pi\rangle$ is a simple matrix iteration with $\cf(\gamma)>\omega$. 
Then, for any $\xi\leq\pi$,
\begin{enumerate}
    \item[(a)]  $\Por_{\gamma,\xi}$ is the direct limit of $\la\Por_{\alpha,\xi}:\alpha<\gamma\ra$, and
    \item[(b)] if $\dot{f}$ is a $\Por_{\gamma,\xi}$-name of a function from $\omega$ into $\bigcup_{\alpha<\gamma}V_{\alpha,\xi}$ then $\dot{f}$ is forced to be equal to a $\Por_{\alpha,\xi}$-name for some $\alpha<\gamma$.  In particular, the reals in $V_{\gamma,\xi}$ are precisely the reals in $\bigcup_{\alpha<\gamma}V_{\alpha,\xi}$.
\end{enumerate}  
\end{lemma}

\begin{theorem}[{\cite[Thm. 10 \& Cor. 1]{M}}]\label{matsizebd}
  Let $\Por=\langle\Por_{\alpha, \xi},\dot{\Qor}_{\alpha,\xi}\,|\,\,\alpha\leq\gamma, \xi\leq\pi\rangle$ be a simple matrix iteration. If $\gamma$ has uncountable cofinality, then $\Por_{\gamma,\pi}$ forces $\non(\Mwf)=\bfrak(\Ed)\leq\cf(\gamma)\leq\dfrak(\Ed)=\cov(\Mwf)$. 
\end{theorem}

To finish this section, we review the followig forcing notion: \emph{Localization forcing} is the poset 
   \[\Loc:=\{\varphi\in\Swf(\omega,\id_\omega):\exists m<\omega\forall i<\omega(|\varphi(i)|\leq m)\}\] 
   ordered by $\varphi'\leq\varphi$ iff $\varphi(i)\subseteq\varphi'(i)$ for every $i<\omega$. Recall that this poset is $\sigma$-linked and that it adds an slalom $\varphi^*$ in $\Scal(\omega,\id_\omega)$ that localizes all the ground model reals in $\omega^\omega$, that is, $x\in^*\varphi^*$ for any $x\in\omega^\omega$ in the ground model.

\section{a connection between $\SNcal$ and $\dfrak_{\kappa\times\lambda}^\lambda$}
In this section  we prove Theorem \ref{mainth1}.

\begin{definition}\label{defsmz}
We say that $X\subseteq2^\omega$ \textit{has strong measure zero} iff for each $f\in\omega^\omega$ there is some $\sigma \in (2^{<\omega})^\omega$ with $\hgt_\sigma=f$ such that $X\subseteq \bigcup_{n<\omega}[\sigma(n)]$. 

Denote $\SNcal:=\{X\subseteq 2^\omega\,|\,X\text{\ has strong measure zero}\}$.
\end{definition}

Denote $\pw_k:\omega\to\omega$ the function defined by $\pw_k(i):=i^k$, and define the relation $\ll$ on $\omega^\omega$ by
$f\ll g \text{\ iff } \forall{k<\omega}(f\circ\pw_k\leq^* g)$. 

For $\sigma \in (2^{<\omega})^\omega$ set
\begin{align*}
[\sigma]_\infty:&=\{x \in 2^\omega\,|\,\exists^{\infty}{n < \omega}^{}(\sigma(n) \subseteq x)\} \\   
                &=\bigcap_{n<\omega} \bigcup_{m \geqslant n}[\sigma(m)]
\end{align*}

\begin{definition}[Yorioka {\cite{Yorioka}}]\label{defidealyorioka}
Let $f\in\omega^\omega$ be an increasing function. Define
\[\Iwf_{f}:=\{X\subseteq 2^{\omega}\,|\,\exists{\sigma \in (2^{<\omega})^{\omega}}(X \subseteq [\sigma]_\infty\text{\ and }h_{\sigma}\gg f )\}.\]
Any family of this form is called a \textit{Yorioka ideal}.
\end{definition}

Yorioka \cite{Yorioka} has proved that $\Iwf_f$ is a $\sigma$-ideal when $f$ is increasing. Moreover, $\SNcal=\bigcap\{\Iwf_f\,|\,f\textrm{\ increasing}\}$. Denote 
$\minnon:=\min\{\non(\Iwf_f)\,|\,f\text{\ increasing}\}$.

\begin{lemma}[{\cite[Lemma 3.7]{Yorioka}}]\label{factCohen}
Let $A$ be a perfect subset of $2^\omega$. Then there some $f\in\omega^\omega$ such that $A\notin\Iwf_f$. 
\end{lemma}

The next definition plays a central role in the main results.

\begin{definition}\label{defsuitmatrix} Let $S$ be a directed partial order. For each increasing function $f\in\omega^\omega$, we say that a family $A^f=\la A_{i}^f\ |\ i\in S\ra$ of subsets of $2^\omega$ is an $\Iwf_f$ \textit{directed system on $S$} it if fulfills the following:
\begin{itemize}
    \item[(I)] $\forall{i\in S}(A_{i}^f\subseteq 2^\omega\textrm{\ is dense  }G_\delta\textrm{\ and\ } \,A_{i}^f\in\Iwf_{f})$;
    \item[(II)] $\forall{i, j\in S}(i\leq j\to A_{i}^f\subseteq A_{j}^f)$ and
    \item[(III)] $\la A_{i}^{f}\,\,|\ i\in S\ra$ is cofinal in $\Iwf_{f}$.
    \end{itemize}
Assume from now on that $S$ has a minimun $i_0$. If $\lambda$ is a cardinal and there is some dominating family $\{ f_\gamma\,|\,\gamma<\lambda\}$ on $\omega^\omega$ such that  $A^{f_\gamma}=\la A_{i}^{f_\gamma}\,|\,i\in S\ra$ is an $\Iwf_{f_\gamma}$ direct system and 
\[ \forall \gamma<\lambda\Big(\bigcap_{\eta<\gamma}A_{i_0}^{f_\eta}\notin\Iwf_{f_\gamma}\Big),\] 
then we say that $\la A^{f_\gamma}\ |\ \gamma<\lambda\ra$ is a $\lambda$ \textit{dominating system on $S$}. 
\end{definition}

\begin{lemma}\label{newversion}
 Let $S$ be a directed partial order and let $\lambda$ be a uncountable cardinal. Assume $\cov(\Mwf)=\dfrak=\lambda$ and that, for any increasing function $f\in\omega^{\omega}$, there is some $\Iwf_f$ directed system on $S$. Then there is some $\lambda$-dominating system on $S$.
\end{lemma}
\begin{proof}
Fix $i_0:=\min(S)$. Let $\la h_\gamma\,|\,\gamma<\lambda\ra$ be a dominating family. For each $\gamma<\lambda$, we denote $A_{i}^{\gamma}:=A_{i}^{f_\gamma}$. We will construct $f_\gamma$ by recursion on $\gamma<\lambda$. Assume that  $\la f_\eta\,|\,\eta<\gamma\ra$ has been constructed. Now, let us assume that $M$ is a transitive model for ZFC such that $|M|<\lambda=\cov(\Mwf)$ and $A_{i_0}^{\eta}$ is coded in $M$ for any $\eta<\gamma$.

Cohen forcing adds a perfect set $P$ of Cohen reals over $M$ (see \cite[Lemma 3.3.2]{BJ}), so  $P\subseteq \bigcap_{\eta<\gamma}A_{i_0}^{\eta}$. Since $P$ is a perfect set, there is some $g\in\omega^\omega$ such that  $P\notin\Iwf_{g}$ by Lemma \ref{factCohen}.

Choose $f_\gamma\in\omega^\omega$ increasing such that $h_\gamma\leq f_\gamma$ and $g\leq f_\gamma$. Then $\Iwf_{f_\gamma}\subseteq\Iwf_g$ and  $P\notin\Iwf_{f_\gamma}$. But $P\subseteq \bigcap_{\eta<\gamma}A_{i_0}^{\eta}$, hence $\bigcap_{\eta<\gamma}A_{i_0}^{\eta}\notin\Iwf_{g}$. 

Clearly, $\la f_\gamma\,|\,\gamma<\lambda\ra$ is a dominating family.
\end{proof}





Below, we shall prove main Theorem \ref{mainth1}(i).

\begin{theorem}\label{lemfamily}
Assume that there is some $\lambda$-dominating system on $S$. Then $\SNcal\leqT \mathbf{D}_{S}^{\lambda}$.



\end{theorem}


\begin{proof}

For $X\in\SNcal$, choose $\Psi_1(X):=G_X\in S^\lambda$ such that $X\subseteq\bigcap_{\gamma<\lambda}A_{G_X(\gamma)}^{\gamma}$ by Definition \ref{defsuitmatrix} (III). Let $F\in S^\lambda$. Note that  $\bigcap_{\gamma<\lambda}A_{F(\gamma)}^{\gamma}\in\SNcal$ because $\bigcap_{\gamma<\lambda}A_{F(\gamma)}^{\gamma}\subseteq A_{F(\gamma)}^{\gamma}$ and $A_{F(\gamma)}^{\gamma}\in\Iwf_{f_\gamma}$. Define $\Psi_2(F):=\bigcap_{\gamma<\lambda}A_{F(\gamma)}^{\gamma}$

Now assume that $\Psi_1(X)\leq F$. Then $G_X(\gamma)\leq F(\gamma)$ for all $\gamma<\lambda$, so by Definition \ref{defsuitmatrix}(II),   $X\subseteq\bigcap_{\gamma<\lambda}A_{G_X(\gamma)}^{\gamma}\subseteq\bigcap_{\gamma<\lambda}A_{F(\gamma)}^{\gamma}$.



\end{proof}

As a consequence:

\begin{corollary}
If there is an $\lambda$-dominating system on $S$ then $\cof(\SNcal)\leq\dfrak_{S}^\lambda$ and $\bfrak(S)=\bfrak_{S}^{\lambda}\leq\add(\SNcal)$.
\end{corollary}

We conclude this section with the proof of Theorem \ref{mainth1}(ii), which will be used in the final section. To prove it, we need the following lemma.

\begin{lemma}\label{one}
Let $\kappa<\lambda$ be infinite cardinals. Assume $\minnon\geq\lambda$ and that there is some $\lambda$-dominating system on $\kappa\times\lambda$. Then, for any $f\in\lambda^\lambda$, there exist  $G\in(\kappa\times\lambda)^\lambda$ and $\{x_\alpha^\gamma:\gamma<\lambda, \alpha<\kappa\}\subseteq2^\omega$ such that 
\begin{enumerate}[(i)]
    \item $\forall\gamma<\lambda(\{x_\alpha^{\gamma'}\,|\, \gamma'\leq\gamma,\alpha<\kappa\}\subseteq A_{G(\gamma)}^{\gamma})$,
    \item $\forall\gamma<\lambda\forall\alpha<\kappa(x_{\alpha}^{\gamma}\in\bigcap_{\gamma'<\gamma}A_{G(\gamma')}^{\gamma'}\smallsetminus A_{\alpha, f(\gamma)}^{\gamma})$, and 
    \item $\forall \gamma<\lambda(f(\gamma)\leq G(\gamma)_1)$.
\end{enumerate}
\end{lemma}
\begin{proof}

We will recursively construct $G(\gamma)\in\kappa\times\lambda$ and $x_{\alpha}^\gamma\in2^\omega$. Assume that we already have $G(\gamma')$ and $x_{\alpha}^{\gamma'}$ for any $\gamma'<\gamma$ and $\alpha<\kappa$. Set $B_\alpha:=A_{\alpha, f(\gamma)}^\gamma\cup\{x_{\beta}^{\gamma'}\,|\,\gamma'<\gamma, \beta<\kappa\}$. Since $\{x_{\beta}^{\gamma'}\,|\,\gamma'<\gamma, \beta<\kappa\}$ has size $<\lambda$, $B_\alpha\in\Iwf_{f_\gamma}$ because $\kappa<\lambda\leq\non(\Iwf_{f_\gamma})$. Then by Definition \ref{defsuitmatrix}, there is some $x_{\alpha}^{\gamma}\in\bigcap_{\eta<\gamma}A_{0,0}^{\eta}\smallsetminus B_\alpha$. Note that $\{x_{\alpha}^{\gamma'}\,|\,\gamma'\leq\gamma, \alpha<\kappa\}\in\Iwf_{f_\gamma}$. Then there must be a $G(\gamma)\in\kappa\times\lambda$ such that $\{x_\alpha^{\gamma'}\,|\,\gamma'\leq\gamma, \alpha<\kappa\}\subseteq A_{G(\gamma)}^{\gamma}$ and $f(\gamma)\leq G(\gamma)_1$. This contruction satisfies the required conditions.
\end{proof}

\begin{lemma}\label{two} 
With the same asumptions as Lemma \ref{one} and 
$G\in(\kappa\times\lambda)^\lambda$  fulfilling its conclusion, if $\beta<\lambda$ and $\delta\leq f(\beta)$ then $\bigcap_{\gamma<\lambda}A_{G(\gamma)}^{\gamma}\not\subseteq A_{\alpha,  \delta}^{\beta}\,\,\textrm{for all}\,\,\alpha<\kappa$.
\end{lemma}
\begin{proof}
By Lemma \ref{one} (i) and (ii), $\{x_\alpha^{\gamma}\, |\,\gamma<\lambda, \alpha<\kappa\}\subseteq \bigcap_{\gamma<\lambda}A_{G(\gamma)}^{\gamma}$ and $x_{\alpha}^{\beta}\notin A_{\alpha, f(\beta)}^{\beta}$. Hence $x_{\alpha}^{\beta}\notin A_{\alpha, \delta}^{\beta}$ because $\delta\leq f(\beta)$.
\end{proof}


\begin{theorem}\label{lemfamilytwo}
Assume $\kappa\leq \lambda$ and that there is some $\lambda$-dominating system on $\kappa\times\lambda$ and $\minnon\geq\lambda$. Then $\Dbf_\lambda^\lambda\leqT\SNcal$. 
\end{theorem}

\begin{proof}
When $\kappa=\lambda$, this is \cite[Thm 3.9]{Yorioka}$\,($with $A_{\alpha}^{\beta}:=A_{\alpha,\alpha}^{f_\beta})$. Assume $\kappa<\lambda$. For $B\in\SNcal$, choose some $F_B\in (\kappa\times\lambda)^\lambda$ such that $B\subseteq\bigcap_{\gamma<\lambda}A_{F_B(\gamma)}^\gamma$. Define $\Psi_2(B):=f_B$ where $f_B(\gamma):=F_B(\gamma)_1$ for every $\gamma<\lambda$. Now fix $f\in\lambda^\lambda$, then by Lemma \ref{one} and Lemma \ref{two} we can find some $G_f\in (\kappa\times\lambda)^\lambda$ fulfilling that, for each $g\in \lambda^\lambda$, for each $\beta<\lambda$, if $g(\beta)\leq f(\beta)$ then \[\bigcap_{\gamma<\lambda}A_{G_f(\gamma)}^{\gamma}\not\subseteq A_{\alpha,  g(\beta)}^{\beta}\,\,\textrm{for all}\,\,\alpha<\kappa.\]
Define $\Psi_1(f):=\bigcap_{\gamma<\lambda}A_{G_f(\gamma)}^{\gamma}$. 

Now assume that $f\not\leq f_B$. We will show that $\bigcap_{\gamma<\lambda}A_{G_f(\gamma)}^{\gamma}\not\subseteq B$. Since $f\not\leq f_B$ choose $\xi<\lambda$ such that $f(\xi)>f_B(\xi)$. Then $\bigcap_{\gamma<\lambda}A_{G_f(\gamma)}^{\gamma}\not\subseteq A_{F_B(\xi)}^{\xi}$. Thus $\bigcap_{\gamma<\lambda}A_{G_f(\gamma)}^{\gamma}\not\subseteq B$ because $B\subseteq\bigcap_{\gamma<\lambda}A_{F_B(\gamma)}^\gamma$.
\end{proof}

As a consequence, we get:

\begin{corollary}
With the same hypothesis as in Lemma \ref{lemfamilytwo}, $\cof(\SNcal)\geq\dfrak_{\lambda}^\lambda$ and $\add(\SNcal)\leq\bfrak_{\lambda}^{\lambda}=\cof(\lambda)$.
\end{corollary}

\section{Model for the cardinal invariants associated with  $\SNcal$}\label{sectionmain}

In this section, we prove Theorem \ref{introtheorem}. But first we need the two following lemmas.

The next lemma shows that a cofinal family in $\Iwf_f$ is produced by a localizating family and a dominating family. 

\begin{lemma}[{\cite[Thm. 3.12]{CM}}]\label{lemconst}
Let $f\in \omega^{\omega}$ be an increasing function. Then there is some definable function $\Psi^f:\omega^{\uparrow\omega}\times \Swf(\omega,\id)\to \Iwf_{f}$ such that, if  
\begin{enumerate}[(i)]
    \item $S\subseteq \Swf(\omega,\id)$ is a localizing family i.e, for any $x\in\omega^\omega$ there is some $\varphi\in S$ such that $x\in^*\varphi$, and 
    \item $D\subseteq \omega^{\uparrow\omega}$ is a dominating family, 
\end{enumerate}
then  $\{\Psi^f(d,\varphi)\,|\,d\in D\textrm{\ and\ }\varphi\in S\}$ is cofinal in $\Iwf_f$.
\end{lemma}



The same proof actually yields:

\begin{lemma}\label{lemconstmatrix}
Let $M$ be a transitive model of ZFC with $f\in\omega^\omega\cap M$ increasing. If $d\in\omega^{\uparrow\omega}$ is dominating over $M$ and $\varphi\in\Scal(\omega,\id)$ is localizing over $M$, then $A\subseteq \Psi^f(d,\varphi)$ for all Borel $A\in\Iwf_f$ coded in $M$.
\end{lemma}

Now, we are ready to prove our main Theorem \ref{mainth2}.

\begin{theorem}\label{maintheorem}
Let $\kappa\leq\lambda$ be regular uncountable cardinals where $\kappa^{<\kappa}=\kappa$ and let $\lambda_1,\lambda_2$ be cardinals such that $\lambda^{<\lambda}=\lambda$, $\lambda\leq\lambda_1$, $\lambda_2^\lambda=\lambda_2$ and $\lambda_1^{\aleph_0}=\lambda_1$. Then there is a cofinality preserving poset that forces 

\vspace{0.3cm}
$(\mathrm{I})$ $\add(\Nwf)=\non(\Mwf)=\kappa\textrm{\ and\ }\cov(\Mwf)=\cof(\Nwf)=\lambda$.
\begin{multline*}
    (\mathrm{II}) \,\,\add(\SNcal)=\cov(\SNcal)=\kappa\leq\non(\SNcal)=\lambda\leq\cof(\SNcal)=\dfrak_\lambda=\dfrak_{\kappa\times\lambda}^\lambda=\lambda_2\\
   \textrm{\ and\ }\cfrak=\lambda_1.
\end{multline*}
\end{theorem}

\begin{proof} \emph{Step 1.} We start with $\Por_0:=\Fn_{<\lambda}(\lambda_2\times\lambda,\lambda)$. $\Por_0$ is $\lambda^+$-cc and $<\lambda$-closed, and thus it preserves cofinalities, and $\Por_0$ forces $\dfrak_{\lambda}=2^\lambda=\lambda_2$. 

\emph{Step 2.} In $V^{\Por_0}$, let $\Por_1:=\Fn_{<\kappa}(\lambda_2\times\lambda,\kappa)$. When $\kappa<\lambda$, $\Por_1$ forces $\dfrak_{\kappa}^{\lambda}=2^\lambda=\lambda_2$ and $\dfrak_{\lambda}=\lambda_2$ because $\Por$ is $\lambda$-c.c (see e.g. \cite[Lemma 2.6]{C}); and if $\lambda=\kappa$, the same is forced by step 1.


\emph{Step 3.}  In $V^{\Por_0\ast\Por_1}$, let $\Por_2:=\Fn_{<\omega}(\lambda_1\times\omega,\omega)$, which forces $\cfrak=\lambda_1$ and $2^\lambda=\max\{\lambda_1,\lambda_2\}$. In particular, $\dfrak_{\kappa\times\lambda}^\lambda=\lambda_2$ because $\Por_2$ is ccc and by Lemma \ref{lemax}.

\emph{Step 4.} We work in $V_{0,0}:=V^{\Por_0\ast\Por_1\ast\Por_2}$. We define the simple matrix iteration of  height $\gamma:=\lambda$ and lenght $\pi:=\lambda\kappa$ where the matrix iteration at each interval of the form $[\lambda\rho,\lambda(\rho+1))$ for each $\rho<\kappa$ is defined as follows. For each $\varepsilon\in[\lambda\rho,\lambda(\rho+1))$, $\varepsilon>0:$
 if  $\varepsilon=\lambda\rho+\xi$ for some $\rho<\kappa$ and $\xi<\lambda$, put $\Delta(\varepsilon)=\xi+1$ and $\Qnm_\varepsilon:=\Loc^{V_{\Delta(\varepsilon), \varepsilon}}$. 

Set $\Por:=\Por_{\lambda,\lambda\kappa}$ and $V_{\alpha,\xi}:=V_{0,0}^{\Por_{\alpha,\xi}}$. We first prove that $\Por$ forces $\kappa\leq\add(\Nwf)$ and $\cof(\Nwf)\leq\lambda$. For each $0<\varepsilon<\lambda\kappa$ denote  by $\varphi^{\varepsilon}\in V_{\Delta(\varepsilon),\,\varepsilon+1}\cap\Scal(\omega,\id)$ the generic slalom over $V_{\Delta(\varepsilon),\varepsilon}$ added by   $\Qnm_{\Delta(\varepsilon),\,\varepsilon}=\Qnm_{\lambda,\varepsilon}=\Loc^{V_{\Delta(\varepsilon), \varepsilon}}$. Hence $V_{\lambda,\lambda\kappa}\models\kappa\leq\add(\Nwf)$ is a consequence of the following 

\begin{claim}[see e.g. {\cite[Claim 5.14]{CM}}]\label{claimb}
In $V_{\lambda,\lambda\kappa}$, each family of reals of size $<\kappa$ is localizated by some $\varphi^\varepsilon$.
\end{claim}

On the other hand, $\{\varphi^\varepsilon\,|\,0<\varepsilon<\lambda\kappa\}$ is a family of slaloms of size $\leq\lambda$ and, by Claim \ref{claimb}, any member of $V_{\lambda,\lambda\kappa}\cap\omega^\omega $ is localizated by some $\varphi^\varepsilon$. Hence $V_{\lambda,\lambda\kappa}\models\cof(\Nwf)\leq \lambda$.

On the other hand, $\Por$ adds $\kappa$-cofinally many Cohen reals by Lemma \ref{matsizebd}, so it forces  $\non(\Mwf)\leq\kappa$. By Theorem \ref{matsizebd}, $\Por$ forces $\cov(\Mwf)=\dfrak(\Ed)\geq \lambda$. Therefore, $\Por$ forces $\kappa=\add(\Nwf)=\non(\Mwf)$ and $\cov(\Mwf)=\cof(\Nwf)=\lambda$. Now, $\Por$ forces:

$\underline{\kappa\leq\add(\SNcal)}$ by (S1) from the introduction;

\vspace{0.15cm}

$\underline{\cov(\SNcal)\leq\kappa}$ because the lenght of the FS iteration on the top has cofinality $\kappa$ and it is well known that such cofinality becomes an upper bound of $\cov(\SNcal)$ (see e.g. \cite[Lemma 8.2.6]{BJ});

\vspace{0.15cm}

$\underline{\non(\SNcal)=\lambda}$ by (S3) from the introduction;

\vspace{0.15cm}

$\underline{\cof(\SNcal)=\lambda_2}$. Let $f\in V_{\lambda,\lambda\kappa}\cap\omega^\omega$ be an increasing function. Then, there are some  $\xi_f<\lambda$ and $\rho_f<\kappa$ such that $f\in V_{\xi_f,\varepsilon^f}$ with $\varepsilon^f=\lambda\rho_f+\xi_f>0$ by Lemma \ref{realint}. 
 
For $\rho<\kappa$ and $\xi<\lambda$ define $\varepsilon^f_{\rho,\xi}:=\lambda(\rho_f+\rho)+\xi_f+\xi$. Let $\dot{\varphi}_{\rho,\xi}^f$ be the $\Por_{\Delta(\varepsilon^f_{\rho,\xi}), \varepsilon^f_{\rho,\xi}+1}$-name of the slalom over $V_{\Delta(\varepsilon^f_{\rho,\xi}),\, \varepsilon^f_{\rho,\xi}}$ and let  $\dot{d}_{\rho,\xi}$ be the $\Por_{\Delta(\varepsilon^f_{\rho,\xi}),\, \varepsilon^f_{\rho,\xi}+1}$-name of some increasing dominating real over $V_{\Delta(\varepsilon^f_{\rho,\xi}),\varepsilon^f_{\rho,\xi}}$ added by $\Qnm_{\varepsilon^f_{\rho,\xi}}$.
Set \[A_{\rho,\,\xi}^{f}:=\Psi^f(\dot{d}_{\rho,\xi}^f,\dot{\varphi}_{\rho,\xi}^f)\, (\textrm{see Figure \ref{matrix}}).\]

\begin{figure}
\begin{center}
\includegraphics[scale=0.65]{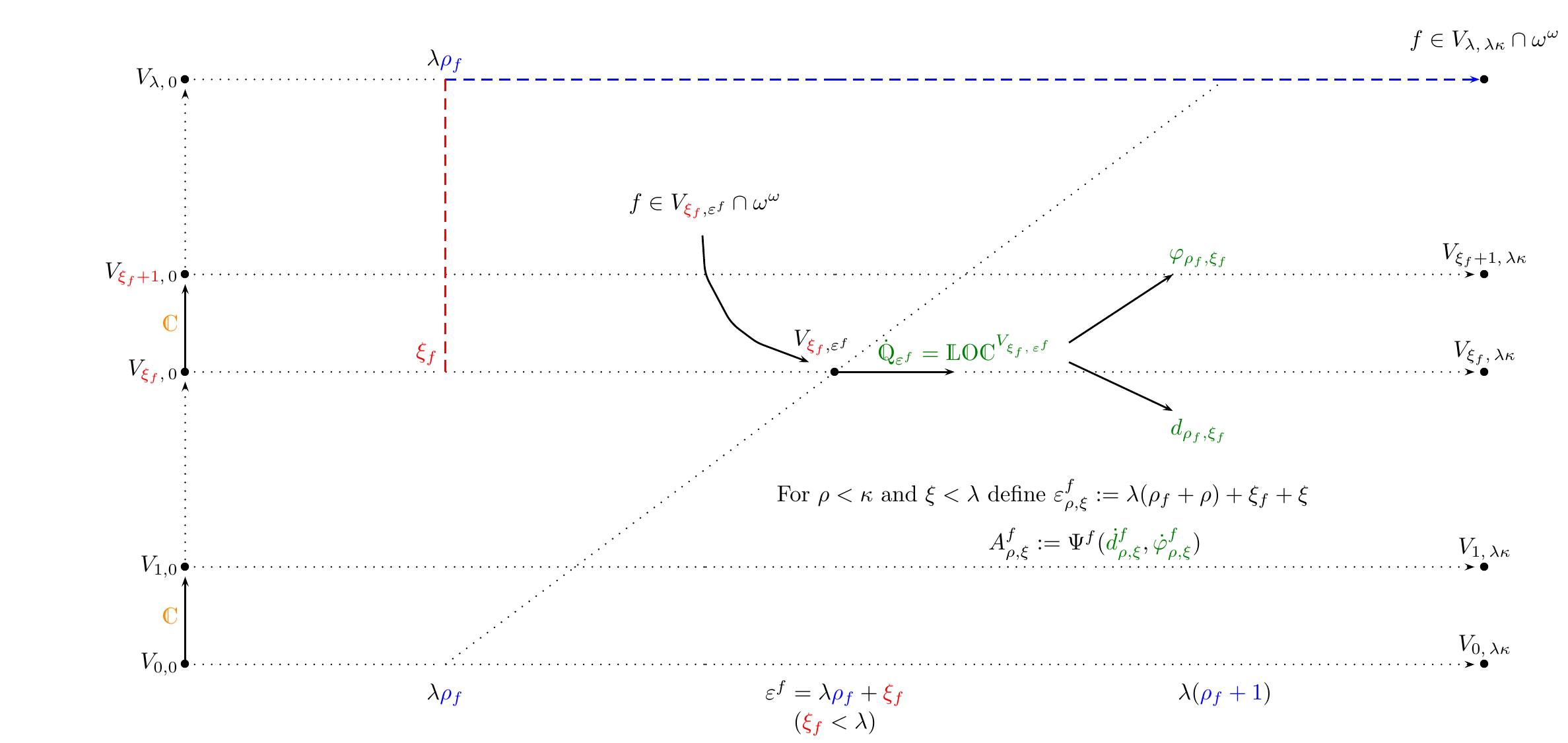}
\caption{Matrix iteration}
\label{matrix}
\end{center}
\end{figure}

\begin{claim}
$\la A_{\rho,\,\xi}^f\,|\,\rho<\kappa\textrm{\ and\ }\xi<\lambda\ra$ is an $\Iwf_f$ directed system.
\end{claim}
\begin{proof}
 It is clear that (I) and (III) follow by Lemma \ref{lemconstmatrix}.  To see (II), note that $\dot{\varphi}_{\rho,\xi}^f$ is an slalom over $V_{\Delta(\varepsilon^f_{\rho,\xi}),\,\varepsilon^f_{\rho,\xi}}$ and $\dot{d}_{\rho,\xi}^f$ is an increasing dominating real over $V_{\Delta(\varepsilon^f_{\rho,\xi}),\,\varepsilon^f_{\rho,\xi}}$, so $A\subseteq A_{\rho,\,\xi}^f$ for any $A\in\Iwf_f$ coded in $V_{\Delta(\varepsilon^f_{\rho,\xi}),\, \varepsilon^f_{\rho,\xi}}$ by Lemma \ref{lemconstmatrix}. In particular, $ A_{\rho',\,\xi'}^f\subseteq A_{\rho,\,\xi}^f$ if $(\rho',\xi')\leq(\rho,\xi)$. 
\end{proof}
We can choose a  $\lambda$-dominating system $\la A^{f_\gamma}\,|\,\gamma<\lambda\ra$ by Lemma \ref{newversion} because $\cov(\Mwf)=\dfrak=\lambda$. Therefore, in $V_{\lambda,\lambda\kappa}$,   $\cof(\SNcal)\leq\dfrak_{\kappa\times\lambda}^\lambda=\lambda_2$  by Theorem \ref{lemfamily}, and  since $\minnon=\lambda$, $\cof(\SNcal)\geq\dfrak_\lambda=\lambda_2$ by Theorem \ref{lemfamilytwo}.
\end{proof}

\section{Open problems} 
In Theorem \ref{lemfamilytwo} we prove $\Dbf_\lambda^\lambda\leqT\SNcal$ assuming  the existence of a $\lambda$-dominating system on $\kappa\times\lambda$. We ask if we could do the same omitting $\kappa$, concretely,
\begin{Questions}
Assume that $\lambda$ is an infinite cardinal and assume that $S$ has no maximum. Do we have $\Dbf_\lambda^\lambda\leqT\SNcal$ whenever the conditions below are satisfied?
\begin{itemize}
    \item[(i)] $\lambda\leqT S$,
    \item[(ii)] there is a $\lambda$-dominating system on $S$, and
    \item[(iii)] $\minnon\geq \lambda$
\end{itemize}
\end{Questions}
More generally, 
\begin{Questions}
Assume that $\lambda$ is an infinite cardinal and assume that $S$ has no maximum. Do we have $\Dbf_\lambda^\lambda\leqT\SNcal$ whenever the conditions below are satisfied?
\begin{itemize}
    \item[(i)] $\Dbf_\lambda^\lambda\leqT\Dbf_S^\lambda$,
    \item[(ii)] there is a $\lambda$-dominating system on $S$, and
    \item[(iii)] $\minnon\geq \lambda$
\end{itemize}
\end{Questions}
Concerning the consistency of the cardinals characteristics associated with $\SNcal$, the following summarises the current open questions.

\begin{Questions}
   Is it consistent with ZFC that 
   \begin{itemize}
       \item[(I)] $\add(\SNcal)<\min\{\cov(\SNcal),\non(\SNcal)\}$?
       \item[(II)] $\add(\SNcal)<\non(\SNcal)<\cov(\SNcal)<\cof(\SNcal)$?
       \item[(III)] $\add(\SNcal)<\cov(\SNcal)<\non(\SNcal)<\cof(\SNcal)$?   
   \end{itemize}
\end{Questions}

Any idea to solve Question (I) in the positive could be used to prove the consistency of (II) and (III), for example using a matrix iteration construction as in this paper.
As mentioned in the introduction, the author with Mej\'ia and Rivera-Madrid solved  Question (II) partially. 

In Theorem \ref{maintheorem} (Thereom \ref{introtheorem}) we answered Question (III) partially, but its consistency still remains open. In this situation, the main issue is  that tools to deal with $\add(\SNcal)$ are still unknown.

\subsection*{Acknowledgments}
The author is very thankful to professor D. Mej\'ia for all the guidance and support
provided during the research that precedes this paper. He offered a lot of his time for discussions that concluded in the results that are presented in this text. The author is also grateful to professor T. Yorioka for his multiple discussions at the Set Theory Seminar in the Departament of Mathematics at Shizuoka University.

{\small
\bibliography{left}
\bibliographystyle{alpha}
}

\end{document}